\documentclass[11pt]{article}
\usepackage{amsfonts,amssymb,amscd,amsmath,enumerate,verbatim,calc, times,url}
\usepackage{amsthm, psfrag,latexsym,epsfig,mdwlist,graphicx}

\theoremstyle{plain}
\newtheorem{theorem}{Theorem}[section]

\newtheorem{proposition}[theorem]{Proposition}
\newtheorem{corollary}[theorem]{Corollary}
\theoremstyle{definition}
\newtheorem{definition}[theorem]{Definition}

\newtheorem{example}[theorem]{Example}

\newtheorem{claim}[theorem]{Claim}

\newcommand{\vs}{v_1,\ldots,v_n}                
\newcommand{\Fs}{F_1,\ldots,F_q}                

\newcommand{\dimn}{ {\rm{dim}} \ }
\newcommand{\rmv}[1]{\setminus \langle #1\rangle}
       
\newcommand{\D}{\Delta}                         
\newcommand{\lcm}{{\mathop{\rm{lcm}}}}          
\newcommand{\ndiv}{\not |}              
\newcommand{\st}{\ | \ }                        
\newcommand{\tuple}[1]{\langle #1 \rangle}      
\newcommand{\void}[1]{}

\newcommand{\XX}{\mathcal{X}}
\newcommand{\erase}[1]{}
\newcommand{\then}{\Longrightarrow}

\newcommand{\cocoa}{\mbox{\rm C\kern-.13em o\kern-.07 em C\kern-.13em o\kern-.15em A}} 
\newcommand{\cocoax}{\mbox{C\kern-.13em o\kern-.07 em C\kern-.13em o\kern-.15em A}} 
\newcommand{\cocoal}{\mbox{\rm C\kern-.13em o\kern-.07 em C\kern-.13em o\kern-.15emA\kern-.1em L}}

\newcommand{\todo}[1]{\vspace{5 mm}\par \noindent
\marginpar{\textsc{ToDo}}
\framebox{\begin{minipage}[c]{0.95 \textwidth}
\tt #1 \end{minipage}}\vspace{5 mm}\par}

\renewcommand{\todo}[1]{}

\newcommand{\idiot}[1]{\vspace{5 mm}\par \noindent
\framebox{\begin{minipage}[c]{0.95 \textwidth}
\tt #1 \end{minipage}}\vspace{5 mm}\par}

\renewcommand{\idiot}[1]{}

\newcommand{\sm}{\setminus}

\renewcommand{\leq}{\leqslant}          
\renewcommand{\geq}{\geqslant}  

\date{}

\author{Sara Faridi\thanks{Department of
Mathematics and Statistics, Dalhousie University, Halifax, Canada, 
faridi@mathstat.dal.ca. 
Research supported by NSERC.}}

\title{\Large \sc Monomial Resolutions Supported By Simplicial Trees}
\begin{document}

\maketitle
\begin{abstract} We explore resolutions of monomial ideals supported by 
simplicial trees. We argue that since simplicial trees are acyclic,
the criterion of Bayer, Peeva and Sturmfels for checking if a
simplicial complex supports a free resolution of a monomial ideal
reduces to checking that certain induced subcomplexes are connected.
We then use results of Peeva and Velasco to show that every simplicial
tree appears as the Scarf complex of a monomial ideal, and hence
supports a minimal resolution. We also provide a way to construct
smaller Scarf ideals than those constructed by Peeva and Velasco.
 \end{abstract}

\section{Introduction} 

The purpose of this paper is to demonstrate that simplicial trees have
the potential to be used as an effective tool in resolutions of
monomial ideals. As first noted by Diane Taylor~\cite{T}, given an
ideal $I$ in a polynomial ring $R$ minimally generated by monomials
$m_1,\ldots,m_q$, a free resolution of $I$ can be given by the
simplicial chain complex of a simplex with $q$ vertices. Most often
Taylor's resolution is not minimal. Bayer, Peeva and
Sturmfels~\cite{BPS} refined Taylor's construction: they provided a
criterion to check if the simplicial chain complex of any simplicial
complex on $q$ vertices is a (minimal) free resolution of $I$
(Theorem~\ref{t:BPS}).

If $\D$ is a simplicial complex with $q$ vertices, the criterion of
Bayer, Peeva and Sturmfels determines if $\D$ supports a free
resolution of $I$ based on whether certain subcomplexes of $\D$ are
acyclic.  The goal of this note is to point out that if the simplicial
complex $\D$ being considered is a simplicial
tree~(Definition~\ref{d:tree}), then all that needs to be checked is
that these subcomplexes are connected. We accomplish this by proving
that simplicial trees are acyclic~(Theorem~\ref{t:collapsible}), and
every induced subcomplex of a simplicial tree is a simplicial
forest~(Theorem~\ref{t:induced}).

 We then use a result of Peeva and Velasco~\cite{PV} to conclude that
 every simplicial tree supports a minimal resolution of a monomial
 ideal. Peeva and Velasco's result is that every simplicial complex
 (other than the boundary of a simplex) is the Scarf complex of some
 monomial ideal, and they give a specific method to build such an
 ideal. We refine their result to describe ideals minimally resolved
 by a Scarf complex, and therefore by a given simplicial tree.

\section{Simplicial trees and some of their properties}

 \begin{definition}[simplicial complex] 
   A \emph{simplicial complex} $\D$ over a set of vertices $V=\{ \vs
   \}$ is a collection of subsets of $V$, with the property that $\{
   v_i \} \in \D$ for all $i$, and if $F \in \D$ then all subsets of
   $F$ are also in $\D$. An element of $\D$ is called a \emph{face} of
   $\D$, and the \emph{dimension} of a face $F$ of $\D$ is defined as
   $|F| -1$, where $|F|$ is the number of vertices of $F$.  The faces
   of dimensions 0 and 1 are called \emph{vertices} and \emph{edges},
   respectively, and $\dimn \emptyset =-1$.  The maximal faces of $\D$
   under inclusion are called \emph{facets} of $\D$. The dimension of
   the simplicial complex $\D$ is the maximal dimension of its facets.
   A subcollection of $\D$ is a simplicial complex whose facets are
   also facets of $\D$; in other words a simplicial complex generated
   by a subset of the set of facets of $\D$.
 \end{definition}

  Suppose $\D$ is a simplicial complex with facets $\Fs$. The
  simplicial complex obtained by \emph{removing the facet} $F_i$ from
  $\D$ is the simplicial complex
 $$\D \rmv{F_i}=\tuple{F_1,\ldots,\hat{F}_{i},\ldots,F_q}.$$ 
 
\begin{definition}[\cite{F} leaf, joint] A facet $F$ of a 
simplicial complex is called a \emph{leaf} if either $F$ is the only
facet of $\D$, or for some facet $G \in \D \rmv{F}$ we have
  $$F \cap (\D \rmv{F}) \subseteq G.$$
Equivalently, a facet $F$ is a leaf of $\D$ if $F \cap (\D \rmv{F})$
is a face of $\D \rmv{F}$.
\end{definition}

Note that it follows immediately from the definition above that a leaf
$F$ must contain at least one \emph{free vertex}; namely a vertex that
belongs to no other facet of $\D$ but $F$.

\begin{definition}[\cite{F} tree, forest]\label{d:tree} A connected 
simplicial complex $\D$ is a \emph{tree} if every nonempty
subcollection of $\D$ has a leaf.  If $\D$ is not necessarily
connected, but every subcollection has a leaf, then $\D$ is called a
\emph{forest}.
\end{definition}


\begin{definition}[induced subcomplex]\label{d:induced} Suppose $\D$ 
is a simplicial complex over a vertex set $V$ and let $\XX \subseteq
V$. The \emph{induced subcomplex on $\XX$}, denoted by $\D_\XX$, is
defined as
$$\D_\XX=\{F\in \D \st F\subseteq \XX \}.$$
\end{definition}

\begin{theorem}\label{t:induced} An induced subcomplex of a simplicial 
tree is a simplicial forest.\end{theorem}

  \begin{proof} Let $\D=\langle F_1,\ldots, F_q\rangle$ be a simplicial
    tree and suppose $\XX=\{x_1,\ldots,x_s\}$ is a subset of the
    vertex set of $\D$. We would like to show that $\D_\XX$ is a
    forest. The facets of $\D_\XX$ are clearly a subset of $\{F_1\cap
    \XX,\ldots,F_q\cap \XX\}$. Let $\Gamma$ be a subcollection of
    $\D_\XX$ consisting of facets $F_{\alpha_1}\cap \XX,\ldots,
    F_{\alpha_r}\cap \XX$.  We need to show $\Gamma$ has a leaf. Since
    $\D$ is a tree, the corresponding subcollection $F_{\alpha_1},
    \ldots, F_{\alpha_r}$ of $\D$ has a leaf $F_{\alpha_i}$ with joint
    $F_{\alpha_j}$. So for every $h \in \{1,\ldots,r\}\sm \{i\}$ we
      have $$F_{\alpha_i}\cap F_{\alpha_h} \subseteq F_{\alpha_j}$$ which implies that
   $$(F_{\alpha_i}\cap \XX) \cap (F_{\alpha_h}\cap \XX) \subseteq
      (F_{\alpha_j}\cap \XX).$$ So $F_{\alpha_i}\cap \XX$ is a leaf of
      $\Gamma$ and therefore $\D_\XX$ is a forest.
  \end{proof}

One property of simplicial trees that we will need is that they are
acyclic. While this can be shown via a direct calculation of
homological cycles and boundaries, we show more: simplicial trees are
collapsible, hence contractible, and therefore acyclic. We refer the
reader to~\cite{B} for more details on these concepts.

\begin{definition}[collapsible simplicial complex] \label{d:collapsible}
Let $\D$ be a simplicial complex and $F'$ be a maximal proper face of
exactly one facet $F$ of $\D$. The complex $\Gamma=\D \sm \{F,F'\}$ is
said to be obtained from $\D$ using an \emph{elementary collapse}. If a
sequence of elementary collapses reduces $\D$ to a single point, then
$\D$ is called \emph{collapsible}.
\end{definition}

Below we use the phrase ``$\D$ collapses to $\D'$'' to imply that the complex 
$\D'$ can be obtained from $\D$ via a sequence of elementary collapses.

\begin{proposition}\label{p:simplex} Let $\D$ be a simplex with facet $F$ and 
let $F'$ be a proper nonempty face of $F$. Then $\D$ collapses to
$\tuple{F'}$. In particular, every simplex is collapsible.
\end{proposition}

    \begin{proof} Suppose $F=\{x_1,\ldots,x_n\}$. We use induction on $n$. 
      The case $n=2$ is clear, since $F'$ would be a point, say
      $\{x_1\}$, and the edge $\{x_1,x_2\}$ clearly collapses to
      $\{x_1\}$.
   
      Suppose $n>2$ and let $F_1,\ldots, F_n$ be the maximal proper
      faces of $F$ where for each $i$, $F_i=F\sm \{x_i\}$. Suppose,
      without loss of generality, $F'\subset F_n$. We perform the
      following elementary collapse on $\D$: $$\D\sm \{F,
      F_1\}=\langle F_2,\ldots,F_n\rangle.$$
      
      \begin{claim}\label{c:simplex} For $i\geq 2$ there is a series of
      elementary collapses taking the complex $\langle F_i,\ldots, F_n
      \rangle$ to the complex $\langle F_{i+1},\ldots, F_n \rangle$.
      \end{claim}

      \begin{proof}[Proof of Claim~\ref{c:simplex}] If $i=2$, then 
       the complex $\D_2=\langle F_2,\ldots,F_n\rangle$ has $F_2\cap
       F_1$ as a maximal proper face of $F_2$ (note that $F_2\cap F_1
       =\{x_3,\ldots,x_n\} \not \subset F_i$ if $i>2$). Now we do the
       elementary collapse
        $$\D_2\sm \{F_2,F_1\cap F_2\}=\langle
       F_3,\ldots,F_n\rangle$$ and we are done.

       Now suppose that we have arrived at $\D_i=\langle F_i,\ldots,
       F_n \rangle$. In what follows we will repeatedly use two basic 
       observations. 
       \begin{enumerate}
       \item[(1)] The maximal proper subfaces of a face
         $F_{i_1,\ldots,i_k}=F_{i_1}\cap F_{i_2}\cap \ldots \cap
         F_{i_k}$ are of the form
        $$F_{i_1,\ldots,i_k,j}=F_{i_1}\cap F_{i_2}\cap \ldots 
        \cap F_{i_k} \cap F_j \mbox{
         where }j \notin \{i_1,\ldots,i_k\} \subseteq
         \{1,\ldots,n\}.$$

       \item[(2)]  Suppose $n\geq i_1>i_2>\cdots>i_s\geq 1$ and 
       $n\geq j_1>j_2>\cdots>j_t\geq 1$.  Then we have 
       $$F_ {i_1,\ldots,i_s} \subseteq F_ {j_1,\ldots,j_t} \iff
         \{i_1,\ldots,i_s\} \supseteq \{j_1,\ldots,j_t\}.$$
       \end{enumerate}

       So the maximal proper faces of $F_i$ that are not contained in
       any of $F_{i+1},\ldots, F_n$
       are $$F_{1,i},F_{2,i},\ldots,F_{i-1,i}.$$ Let $\D_{i+1}=\langle
       F_{i+1},\ldots, F_n \rangle$. Using~(1) and~(2) above we perform
       the repeated elementary collapses 
        $$\begin{array}{rl}
        \D_{i,1}=&\D_i \sm \{F_i,F_{i,1}\}
          = \langle F_{i,2},\ldots,F_{i,i-1}\rangle \cup \D_{i+1}\\
          &\\
         \D_{i,2,1}=&\D_{i,1} \sm \{F_{i,2}, F_{i,2,1}\}
          = \langle
         F_{i,3},\ldots,F_{i,i-1}\rangle \cup \D_{i+1}\\
         &\\
         \D_{i,3,1}=&\D_{i,2,1} \sm \{F_{i,3}, F_{i,3,1}\}
         =\langle F_{i,3,2}, F_{i,4},
         \ldots,F_{i,i-1}\rangle \cup \D_{i+1}\\
         &\\     
         \D_{i,3,2,1}=&\D_{i,3,1} \sm \{F_{i,3,2}, F_{i,3,2,1}\}
         = \langle F_{i,4},\ldots,F_{i,i-1}  \rangle \cup \D_{i+1}\\
         &\\
         \D_{i,4,1}=&\D_{i,3,2,1} \sm \{F_{i,4}, F_{i,4,1}\}
         = \langle F_{i,4,2}, F_{i,4,3},
         F_{i,5},\ldots,F_{i,i-1}\rangle \cup \D_{i+1}\\
         &\\
         \vdots\ \ \ \ \ \ \ &\\ 
         \D_{i,\ldots,1}=&\D_{i,i-1,i-2,\ldots,3,1}\sm \{F_{i,\dots,2},F_{i,\dots,1}\}
         = \D_{i+1} 

       \end{array}$$
        
        \end{proof}

       It now follows from repeated applications of Claim~\ref{c:simplex}
       that $\D$ collapses to $\D_n=\langle F_n \rangle$, which is a
       simplex on $n-1$ vertices. If $F'=F_n$, we are done, and if
       not the induction hypothesis implies that $\D_n$ collapses to
       $\langle F' \rangle$ via a series of elementary collapses.
    \end{proof}

\begin{theorem}\label{t:collapsible} Simplicial trees are collapsible, and
 therefore contractible and acyclic.
\end{theorem}

   \begin{proof} We prove this by induction on the number $q$ of facets
    of a simplicial tree $\D$. If $q=1$ the statement follows from
    Proposition~\ref{p:simplex}. Suppose $q>1$ and let $F$ be a leaf
    of $\D$ with joint $G$. Let $F'=F\cap G$. We know by
    Proposition~\ref{p:simplex} that $\langle F \rangle$ reduces to
    $\tuple{F'}$ via a series of elementary collapses. Moreover, the
    faces being eliminated in in each of the collapses are not faces
    of $\D\rmv{F}$, since they are not faces of $F'=F\cap
    \D\rmv{F}$. Therefore, all the elementary collapses that reduce
    $\tuple{F}$ to $\tuple{F'}$ are elementary collapses on $\D$ that
    reduce $\D$ to $\D\rmv{F}$. The latter is a tree with $q-1$
    facets, and hence collapsible by the induction hypothesis.

    All collapsible complexes are contractible so the rest of the statement 
    follows directly.
   \end{proof}

\section{Resolutions by trees} 

We now review monomial resolutions as described by Bayer, Peeva and
Sturmfels~\cite{BPS} and show how simplicial trees fit in that
picture. The construction in~\cite{BPS} considers a monomial ideal $I$
in a polynomial ring $S$ over a field, where $I$ is minimally
generated by monomials $m_1,\ldots,m_t$. If $\D$ is a simplicial
complex on $t$ vertices, once can label each vertex of $\D$ with one
of the generators of $m_1,\ldots,m_t$ and each face with the least
common multiple of the labels of its vertices. If $m$ is a monomial in
$S$, let $\D_m$ be the subcomplex of $\D$ induced on the vertices of
$\D$ whose labels divide $m$.

\begin{theorem}[Bayer, Peeva, Sturmfels~\cite{BPS}]\label{t:BPS} Let $\D$ be a 
simplicial complex labeled by monomials $m_1,\ldots,m_t \in S$, and
let $I = (m_1,\ldots,m_t)$ be the ideal in $S$ generated by the vertex
labels. The chain complex ${\mathcal{C}}(\D) = {\mathcal{C}}(\D ; S)$
of $\D$ is a free resolution of $S/I$ if and only if the induced
subcomplex $\D_m$ is empty or acyclic for every monomial $m \in
S$. Moreover, ${\mathcal{C}}(\D)$ is a minimal free resolution if and
only if $m_A \neq m_{A'}$ for every proper subface $A'$ of a face $A$.
\end{theorem}

Note that we can determine whether ${\mathcal{C}}(\D)$ is a resolution
just by checking the vanishing condition for monomials that are least
common multiples of sets of vertex labels.

Combinatorially what Theorem~\ref{t:BPS} is saying that the Betti
vector of $S/I$ is bounded by the $f$-vector of an eligible $\D$:
\begin{eqnarray}\label{e:bettif}
\beta(S/I)=(\beta_0(S/I),\ldots,\beta_q(S/I))\leq
(f_0(\D),\ldots,f_q(\D))=\mathbf{f}(\D).
\end{eqnarray}
with equality holding if some extra conditions are satisfied.

 We now turn our attention back to simplicial trees. If the $\D$ under
consideration in Theorem~\ref{t:BPS} is a tree, then we can show the following.

\begin{theorem}[Resolutions via simplicial trees]\label{t:main} 
Let $\D$ be a simplicial tree labeled by monomials $m_1,\ldots,m_t \in
S$, and let $I = (m_1,\ldots,m_t)$ be the ideal in $S$ generated by
the vertex labels. The chain complex ${\mathcal{C}}(\D) = {\mathcal{C}}(\D ;
S)$ is a free resolution of $S/I$ if and only if the induced
subcomplex $\D_m$ is connected for every monomial $m$.
\end{theorem}

   \begin{proof} By Theorem~\ref{t:induced} every induced subcomplex of
    $\D$ is a forest. By Theorem~\ref{t:collapsible} forests are
     acyclic in all but possibly the $0$-th reduced homology, that is
     they may not be connected. This proves the theorem.
   \end{proof}

The strength of Theorem~\ref{t:main} is in that it reduces the
question of whether a simplicial complex resolves an ideal to checking
whether some of its induced subcomplexes are connected.

One type of question one could then ask is given a tree $\D$, what
ideals could it resolve? Our first example displays this line of
questioning.

\begin{example}\label{e:gen} Let $\D$ be the simplicial tree 
below on $4$ vertices, which we have labeled with monomials
$m_1,\ldots,m_4$.
\[\includegraphics[width=2in, height=1in]{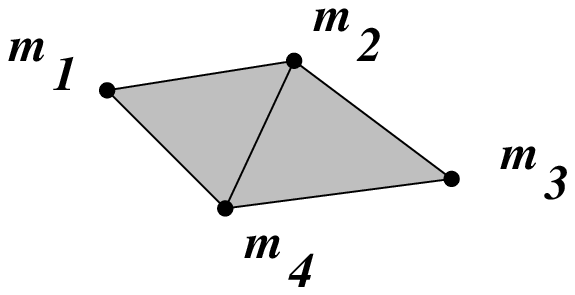}\]

The only induced subcomplex of $\D$ that is not connected is the one induced
on the vertices labeled $m_1$ and $m_3$, so by Theorem~\ref{t:main}
for $I=(m_1,m_2,m_3,m_4)$ to be resolved by $\D$ we need to have
$$m_2 | \lcm(m_1,m_3) \mbox{ \ \ \  or \ \ \ } m_4 | \lcm(m_1,m_3).$$
\end{example}

A more concrete example using the same complex comes next.

\begin{example}\label{e:scarf} The ideal $I=(xy^2,yz,xz^2,zu)$ can be
resolved by $\D$. 
\[\includegraphics[width=2in, height=.8in]{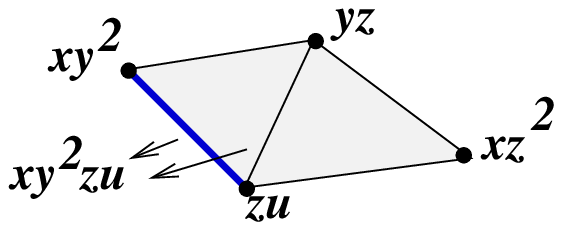}\]
However $\beta(S/I)=(4,4,1)\lneq (4,5,2)=\mathbf{f}(\D)$ so the
  resolution is not minimal.  We try to make it minimal by removing
  the faces with equal labels.
\[\includegraphics[width=2in, height=.8in]{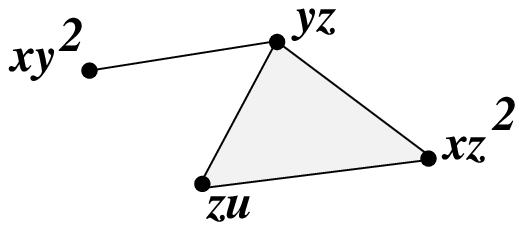}\]
Note that the resulting complex is also a simplicial tree satisfying
the conditions of Theorem~\ref{t:main} and whose $f$-vector is 
$(4,4,1)$. It therefore minimally resolves $S/I$.

\end{example}

\section{Scarf complexes and Scarf ideals}

 We now come to the question of which monomial ideals can be
 (minimally) resolved by a simplicial tree. It is known from work of
 Velasco~\cite{V} that there are classes of monomial ideals whose
 resolutions are not supported by any simplicial complex. However, most
 simplicial complexes, and all simplicial trees do appear as
 \emph{Scarf complexes} of some monomial ideal.  Given a monomial
 ideal its Scarf complex is a subcomplex of its Taylor complex with
 the same labeling and with the added condition that if a face has the
 same label as another face, neither face can appear in the Scarf
 complex. The last simplicial complex appearing in
 Example~\ref{e:scarf} is a Scarf complex of the ideal $I$ in that
 example.

  By construction, if the Scarf complex resolves an ideal, it does so
  minimally.  Moreover most simplicial complexes appear as the Scarf
  complex of some monomial ideal.

\begin{theorem}[\cite{PV}, \cite{Ph}] Let $\D$ be a 
simplicial complex on $r$ vertices.
\begin{enumerate}
\item $\D$ is the Scarf complex of a monomial ideal if and only if
  $\D$ is not the boundary of a simplex on $r$ vertices.
\item $\D$ minimally resolves a monomial ideal if and only if $\D$ is
  acyclic.
\end{enumerate}
\end{theorem}

Since simplicial trees are acyclic, it immediately follows that 

\begin{corollary}\label{c:tree-scarf} Every simplicial tree 
is the Scarf complex of a monomial ideal $I$ and supports a minimal
resolution of $I$.
\end{corollary}

An ideal (minimally) resolved by its Scarf complex is called a
\emph{Scarf ideal}. Given an eligible simplicial complex $\D$ with
vertices labeled $1,\ldots,n$, Peeva and Velasco in~\cite{PV} build a
Scarf ideal $J_\D$ using the following steps. Define a variable
$x_\sigma$ corresponding to each face $\sigma$ of $\D$. In the
polynomial ring generated by all these variables, define the ideal
$J_\D$ whose generators are enumerated by the vertices of $\D$, and
for every given vertex $v$ of $\D$, the corresponding monomial
generator is the product of all $x_\sigma$ where $v \notin \sigma$. In
short
\begin{eqnarray}\label{e:j}
 J_\D=(\prod_{\stackrel{\sigma \in \D}{v\notin \sigma}}^{} x_\sigma
 \st v=1,\ldots,n)=(m_1,\ldots,m_n).
\end{eqnarray}

The ideal $J_\D$ defined above is generated by rather large
monomials. In what follows we will demonstrate that one can shave off
some variables in each monomial to reduce the size of the generator
and still have a Scarf ideal of $\D$.

Suppose $\D$ is a simplicial complex with vertices labeled
$1,\ldots,n$. And for each vertex $v$ let $A_\D(v)$ be the set of
facets of $\D$ that do not contain $v$, and let $B_\D(v)$ be the set
of facets of $\D$ that do contain $v$. With variables labeled as
described above, define the ideal
$$J'_\D=(m'_1,\ldots,m'_n)$$ where  
\begin{eqnarray}\label{e:j'}
m'_v=\sqrt{\prod_{G \in B_\D(v)}x_{G\sm\{v\}}\ \prod_{F \in A_\D(v)} x_F
\prod_{\stackrel{\sigma\subset F}{\scriptscriptstyle |\sigma|=|F|-1}}x_\sigma} 
\mbox{\ \ \ \ \  for \ \ \ }
v=1,\ldots,n.
\end{eqnarray}

It is clear that the $m'_v \st m_v$ for all $v$.

\begin{proposition}\label{p:scarf} Let $\D$ be a simplicial complex which is 
not the boundary of an $n$-simplex and let $J'_\D$ be the ideal
described in (\ref{e:j'}).
\begin{enumerate} 
\item $\D$ is the Scarf complex for $J'_\D$. 
\item If $\D$ is acyclic (and in particular if $\D$ is a simplicial
  tree) then $J'_\D$ is a Scarf ideal.
\end{enumerate}
\end{proposition}

  \begin{proof} We first show that $J'_\D$ has no redundant generators. 
    Suppose that we have $m'_i \ |\ m'_j$ and $i\neq j$. 

    Clearly $A_\D(i)\subseteq A_\D(j)$.  If $G \in B_\D(i)$, then
    $G\sm \{i\}$ can only be a maximal proper face of a facet in
    $A_\D(j)$; otherwise $H=\{j\}\cup G\sm \{i\} \in B_\D(j)$ and $i
    \notin H$, therefore $H\in A_\D(i)\subseteq A_\D(j)$ which is a
    contradiction since $j \in H$. In particular $G \in A_\D(j)$.

    We have shown that $$A_\D(i) \cup B_\D(i) \subseteq A_\D(j).$$
    This implies that all facets of $\D$ belong to $A_\D(j)$ and hence
    $j$ is not in any facet of $\D$; a contradiction.

     So we can label the vertices of $\D$ with the monomials
     $m'_1,\ldots, m'_n$, where the labeling is consistent with
     $m_1,\ldots,m_n$ as in (\ref{e:j}). Next we have to make sure
     that $\D$ is a Scarf complex of $J'_\D$.  For this purpose and
     what follows, the next claim will be useful.

      \bigskip

     \begin{claim}\label{c:scarf} Suppose $\sigma=\{u_1,\ldots,u_s\}$ and
     $\tau=\{v_1,\ldots,v_t\}$ are two faces of the simplex on
     $\{1,\ldots, n\}$. Then
     $$\lcm(m'_{u_1},\ldots,m'_{u_s})=\lcm(m'_{v_1},\ldots,m'_{v_t})
     \iff
     \lcm(m_{u_1},\ldots,m_{u_s})=\lcm(m_{v_1},\ldots,m_{v_t}).$$
     \end{claim}

     \begin{proof}[Proof of Claim~\ref{c:scarf}] For ease of argument we
     label the above $\lcm$s from the left to right with the symbols
     $M'_\sigma$, $M'_\tau$, $M_\sigma$ and $M_\tau$,
     respectively. Now suppose $M'_\sigma=M'_\tau$. Then it follows
     directly that $M_\sigma=M_\tau$. Conversely, suppose
     $M_\sigma=M_\tau$. Then, in particular, we have
      $$\bigcup_{i=1}^{s} A_\D(u_i)=\bigcup_{i=1}^{t} A_\D(v_i)$$ so 
      all the factors $x_F$ where $F$ is a facet of $\D$ are the same
      in both monomials $M'_\sigma$ and $M'_\tau$, as well as all
      $x_\sigma$ for maximal proper faces $\sigma$ of such $F$. So we
      only have to worry about terms of the form $x_{G\sm \{j\}}$ for
      a facet $G$ of $\D$ that contains $j$. Suppose $x_{G\sm
      \{{u_h\}}} \ |\ M'_\sigma$. If $x_G$ appears in $M'_\sigma$,
      we are done, as $G\sm \{{u_h\}}$ is a maximal proper face of $G$
      which appears as a label in $M'_\tau$ as well. If not, we
      conclude that $u_1,\ldots,u_s,v_1,\ldots,v_t \in G$, which means
      that $\sigma$ and $\tau$ are both faces of $\D$ with the same
      $\lcm$s; a contradiction as $\D$ is a Scarf complex of $J_\D$.
       \end{proof}

      The statement we just proved implies that $\D$ is the Scarf
      complex of $J'_\D$, as it is the Scarf complex of $J_\D$.
 
      We now show that if $\D$ is acyclic, then it supports a
      (minimal) resolution of $J'_\D$. So we need to show that for any
      set of vertices $u_1,\ldots,u_s$ of $\D$, the induced
      subcomplex on the vertex set $$\XX=\{j \st m'_j \st
      \lcm(m'_{u_1},\ldots,m'_{u_s})\}$$ is acyclic. Notice that
      $$\lcm(m'_j \st j \in \XX)= \lcm (m'_{u_1},\ldots,m'_{u_s})$$
      which by Claim~\ref{c:scarf} is equivalent to
      $$\lcm(m_j \st j \in \XX)= \lcm (m_{u_1},\ldots,m_{u_s})
      \mbox{ and }\XX=\{j \st m_j \st  \lcm(m_{u_1},\ldots,m_{u_s})\}.$$ So the
      induced subcomplex $\D_\XX$ is the same under both labelings (by
      $J_\D$ and $J'_\D$), and is therefore acyclic.
  \end{proof}

  We demonstrate all this via an example.

  \begin{example} For the complex $\D$ below, 
$\beta(J_\D)=(4,4,1)=\beta(J'_\D)=\mathbf{f}(\D)$.
 \[\includegraphics[width=1.5in, height=.8in]{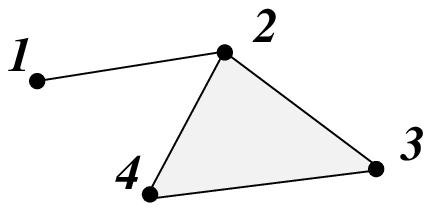}\]
$$J_\D=(\mathbf{x_2}x_3x_4\mathbf{x_{23}x_{24}x_{34}x_{234}}, 
\mathbf{x_1}x_3x_4\mathbf{x_{34}}, 
\mathbf{x_1x_2}x_4\mathbf{x_{12}x_{24}},
\mathbf{x_1x_2}x_3\mathbf{x_{12}x_{23}})$$
$$\downarrow$$
$$J'_\D=(x_2x_{23}x_{24}x_{34}x_{234}, 
x_1x_{34}, 
x_1x_2x_{12}x_{24},
x_1x_2x_{12}x_{23})$$

 \end{example}

Computational evidence has shown that many ideals ``in-between''
$J_\D$ and $J'_\D$ can be resolved by $\D$, though not all of them, as
indicated in Example~\ref{e:montreal}. Given a vertex $v$ of $\D$, we
know that
\begin{eqnarray}\label{e:j''}
m_v=\prod_{\stackrel{\sigma \in \D}{v\notin \sigma}}^{} x_\sigma=m_v'' m_v'
\end{eqnarray}
where by $m_v''$ we are denoting the product of the $x_\sigma$ that do
not appear in $m_v'$.

\begin{proposition}\label{p:scarf} Let $\D$ be a simplicial complex on $n$
vertices which is not the boundary of a simplex. For a vertex $v$ of
$\D$, let the monomials $m_v$, $m'_v$ and $m''_v$ be as defined in
(\ref{e:j}), (\ref{e:j'}) and (\ref{e:j''}), respectively, and suppose
$h_v$ is a monomial such that $h_v \st m''_v$.  Let $I$ be the
monomial ideal $$I=(h_1m'_1,\ldots,h_nm_n').$$ Then the Scarf complex
$\Gamma$ of $I$ has $n$ vertices and contains $\D$ as a subcomplex.
\end{proposition}

   \begin{proof} First we have to show that $I$ has no 
  redundant generators.  Consider two monomials $h_im'_i$ and
  $h_jm'_j$ for some $i\neq j$. We have proved before that $m'_i \ndiv
  m'_j$, so there there are two possibilities:

  \begin{enumerate}

      \item There is $F\in A_\D(i)$ such that $F\notin A_\D(j)$
        (therefore $j \in F$), in which case $x_F \ndiv m_j$, and
        therefore $h_im'_i \ndiv h_jm'_j$; a contradiction.

      \item $A_\D(i) \subseteq A_\D(j)$, in which case there is $G\in
        B_\D(i)$ such $x_{G\sm\{i\}} \ndiv m'_j$, so $G \notin
        A_\D(j)$ and therefore $j \in G$ which implies that $j \in
        G\sm \{i\}$ so $x_{G\sm \{i\}} \ndiv m_j$, and therefore $h_im'_i \ndiv
        h_jm'_j$.
     \end{enumerate}
    
    This shows that $h_1m'_1,\ldots,h_nm_n'$ is a minimal generating
    set for $I$.

    Let $\Gamma$ be the Scarf complex of $I$ and suppose
    $\sigma=\{u_1,\ldots,u_s\}$ and $\tau=\{v_1,\ldots,v_t\}$ are two
    faces of the simplex on $\{1,\ldots, n\}$ with the same labels:
     $$\lcm(h_{u_1}m'_{u_1},\ldots,h_{u_s}m'_{u_s})=
    \lcm(h_{v_1}m'_{v_1},\ldots,h_{v_t}m'_{v_t}).$$ Suppose $u_i\notin
    \{v_1,\ldots, v_t\}$ for some $i$, then we have $h_{u_i}m'_{u_i}
    \st \lcm(h_{v_1}m'_{v_1},\ldots,h_{v_t}m'_{v_t})$. So all
    variables labeled by facets in $A_\D(u_i)$, their maximal proper
    faces, and by $G\sm \{u_i\}$ for $G \in B_\D(u_i)$ already appear
    in $\lcm(h_{v_1}m'_{v_1},\ldots,h_{v_t}m'_{v_t}) \st
    \lcm(m_{v_1},\ldots,m_{v_t})$. Therefore
    $$m_{u_i} \st \lcm(m_{v_1},\ldots,m_{v_t}) \then \lcm(m_{u_i},
    m_{v_1},\ldots,m_{v_t})= \lcm(m_{v_1},\ldots,m_{v_t}).$$ Since
    $\D$ is the Scarf complex for $J_\D$, this implies that $\tau
    \notin \D$. Similarly we have $\sigma \notin \D$. This proves that
    the Scarf complex $\Gamma$ of $I$ contains $\D$.

 \end{proof}

 Below is an example demonstrating that $\Gamma$ may not be equal to
 $\D$, even though they are quite often equal.

  \begin{example}\label{e:montreal} For the complex $\D$ below 
 \[\includegraphics[width=1.5in, height=1in]{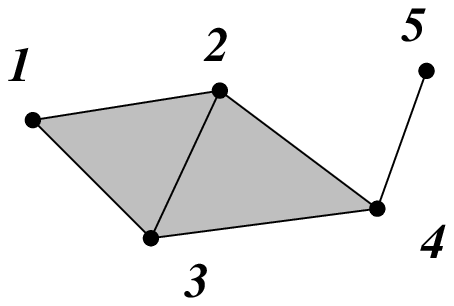}\]
we have $J_\D=(m_1,\ldots,m_5)$ and $J'_\D=(m'_1,\ldots,m'_5)$ where
$$\begin{array}{llll}
m_1=&x_2x_3 m'_1&m'_1=&x_{23}x_{24}x_{34}x_{234}x_{4}x_{5}x_{45}\\
m_2=&x_{1}x_{3}m'_2
&m'_2=&x_{13}x_{34}x_{4}x_{5}x_{45}\\
m_3=&x_{1}x_{2}m'_3
&m'_3=&x_{12}x_{24}x_{4}x_{5}x_{45}\\
m_4=&x_{1}x_{2}x_{3}m'_4
&m'_4=&x_{12}x_{13}x_{23}x_{123}x_{5}\\
m_5=&x_{1}x_{2}x_{3}m'_5
&m'_5=&x_{12}x_{13}x_{23}x_{123}x_{24}x_{34}x_{234}x_{4}\\
\end{array}
$$
In this case, $\beta(S/J_\D)=\beta(S/J'_\D)=\mathbf{f}(\D)=(5,6,2)$ as
expected (though $J_\D$ and $J'_\D$ have different graded Betti numbers).

Now consider the ideal $I=(m'_1,m'_2,m'_3,x_1m'_4,m'_5)$. We have 
 $\beta(S/I)=(5,7,3)$ and the (acyclic) Scarf
complex of $I$ is
 \[\includegraphics[width=1.5in, height=1in]{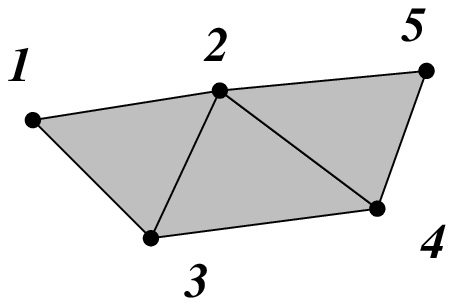}\]
which contains $\D$ as a subcomplex. 

It is worth noting that only very low degree choices of $h_v$ will
give strictly larger Scarf complexes. That is, given an acyclic
simplicial complex, one can find a whole class of Scarf ideals for it
by making appropriate (large enough) choices for the monomials $h_v$.
 \end{example}

There are many questions that naturally follow from this work, answers
to which would greatly contribute to understanding monomial
resolutions. For example, can one describe classes of monomial ideals
resolved by a given tree? What roles do localization, removal of facets
and other such operations that preserve forests play on Scarf ideals?
Can one describe classes of complexes (trees) resolving a given
monomial ideal? 


\end{document}